\documentclass[a4paper,11pt]{article}
 \usepackage[english]{babel}
 \usepackage{graphicx}
 \usepackage[utf8]{inputenc}
 \usepackage[T1]{fontenc}
 \usepackage{lmodern}
 \usepackage[normalem]{ulem}
 \usepackage{verbatim}
 \usepackage{bbm}
 \usepackage{ntheorem}
 \usepackage{stmaryrd}
 \usepackage{amsmath}
 \usepackage{amssymb}
 \usepackage{dsfont}
\usepackage{hyperref}
 \usepackage{manfnt}
  \usepackage{manfnt}
  \usepackage[font=small]{caption}
\theoremheaderfont{\scshape}
\theoremseparator{.}
\newtheorem{theorem}{Theorem}[section]
\newtheorem{thm}[theorem]{Theorem}

\newtheorem{lem}[theorem]{Lemma}

\newtheorem{prop}[theorem]{Proposition}
\newtheorem{coro}[theorem]{Corollary}
\newtheorem{fact}[theorem]{Fact}
\theoremheaderfont{\itshape}
\theorembodyfont{\upshape}
\theoremseparator{.}
\newtheorem{rem}[theorem]{Remark}

\def\sqw{\hbox{\rlap{\leavevmode\raise.3ex\hbox{$\sqcap$}}$%
\sqcup$}}

\newcommand{\N}{\ensuremath{{\mathbb N}}}
\newcommand{\Z}{\ensuremath{\mathbb Z}}

\newcommand{\R}{\ensuremath{\mathbb R}}

\newcommand{\Prime}{\mathcal{P}}

\newenvironment{example}[1][Example.]{\begin{trivlist}
\item[\hskip \labelsep {\textit{#1}}]}{\end{trivlist}}

\newcommand{\Zhat}{\hat \Z}

\newenvironment{proof}{  
    \vspace*{-.4em}  {\it Proof.}%
}{
    \hfill\sqw\vspace*{.5em}
}

\newenvironment{proofdomisto}{  
    \vspace*{-.4em}  {\it Proof of Proposition~\ref{prop:stochordering}.}%
}{
    \hfill\sqw\vspace*{.5em}
}

\newenvironment{proofpropcoprime}{  
    \vspace*{-.4em}  {\it Proof of Proposition~\ref{prop:coprime}.}%
}{
    \hfill\sqw\vspace*{.5em}
}

\newenvironment{proofpropgcd}{  
    \vspace*{-.4em}  {\it Proof of Proposition~\ref{prop:gcd}.}%
}{
    \hfill\sqw\vspace*{.5em}
}

\newenvironment{proofthmaff}{  
    \vspace*{-.4em}  {\it Proof of Proposition~\ref{thm:gcdaffine}.}%
}{
    \hfill\sqw\vspace*{.5em}
}

\newenvironment{proofcoraff}{  
    \vspace*{-.4em}  {\it Proof of Corollary~\ref{coro:gcdaffine}.}%
}{
    \hfill\sqw\vspace*{.5em}
}

\newcommand{\defini}{\textbf}
\newcommand{\stand}{\mathfrak{X}_0}
\author{Sébastien \textsc{Martineau}\footnote{Université Paris-Sud,  {\href{mailto:sebastien.martineau@u-psud.fr}{\nolinkurl{sebastien.martineau@u-psud.fr}}}}}
\title{On coprime percolation, the visibility graphon, and the local limit of the \textsc{gcd} profile}
\begin{document}
\maketitle

\renewcommand{\gcd}{\textsc{gcd}}

\begin{abstract}
Colour an element of $\Z^d$ white if its coordinates are coprime and black otherwise. What does this colouring look like when seen from a ``uniformly chosen'' point of $\Z^d$? More generally, label every element of $\Z^d$ by its \textsc{gcd}: what do the labels look like around a ``uniform'' point of $\Z^d$? We answer these questions and generalisations of them, provide results of graphon convergence, as well as a ``local/graphon'' convergence. One can also investigate the percolative properties of the colouring under study.
\end{abstract}

\vfill

\begin{center}
\includegraphics[width=9.6cm]{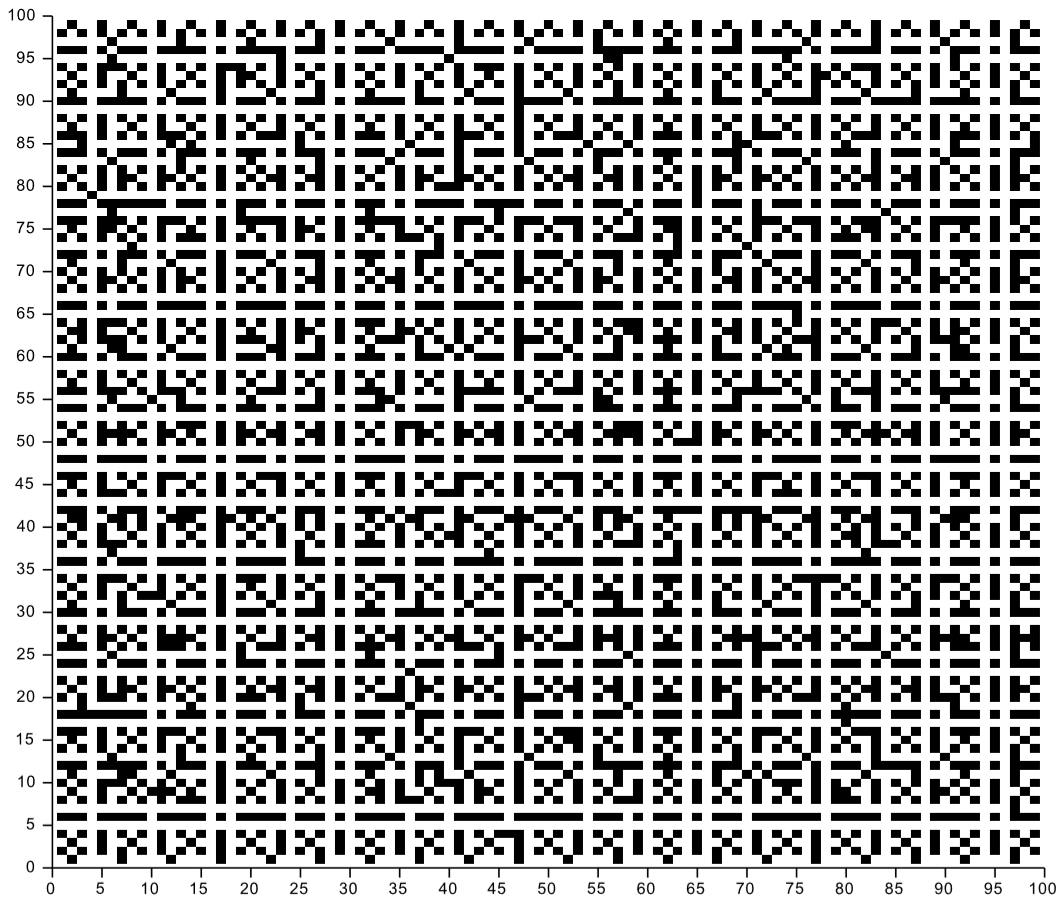}
\label{figurep}
\end{center}

\vfill

\newpage

Colour an element of $\Z^d$ white if its coordinates are coprime and black otherwise. What does this colouring look like? For $d=2$, this question was investigated in \cite{vardi}. The starting point of the present paper is the following question: what does this colouring look like \emph{when seen from a ``uniformly chosen'' point of $\Z^d$}? An answer has already been formulated in \cite{pleasantshuck} but, from the perspective adopted in the current paper, our vocabulary, techniques and results are more satisfactory. See the figure on page~\pageref{figurep}.

A more general version of this question goes as follows: if one labels every element of $\Z^d$ by its \textsc{gcd}, what do the labels look like around a ``uniform'' point of $\Z^d$? We answer this question and generalisations of it, provide results of graphon convergence, as well as a ``local/graphon'' convergence\footnote{Essentially, the same vertex-set will at the same time be endowed with some structure of sparse graph and some structure of dense graph.}. By using previous work of Vardi \cite{vardi}, we can also investigate the percolative properties of the colouring under study.

\section{Introduction}

In this paper, the set $\mathbb{N}=\{0,1,2,3,4,\dots\}$ is taken to contain 0. The set of positive integers will be denoted by $\N^\star=\{1,2,3,4,\dots\}$.

Let $d$ denote a positive integer. It is well-known that the probability that $d$ numbers chosen independently and uniformly in $\llbracket 1,N \rrbracket$ are globally coprime converges to $1/\zeta(d)$ when $N$ goes to infinity \cite{dirichlet, cesaro1, cesaro2, sylvester83}. Recall that on $[1,\infty)$, the \defini{Euler--Riemann $\zeta$ function} is defined by
$$
\forall s\in[1,\infty),~\zeta(s):=\sum_{n\geq 1}n^{-s}=\prod_{p\in \Prime}\frac{1}{1-p^{-s}}\in [1,\infty],
$$
where $\Prime=\{2,3,5,7,11,\dots\}$ denotes the set of prime numbers.
More generally, one has the following result (Theorem~459 in \cite{hardywright}).

\begin{equation}
\label{eq:classical}
\parbox{10.5cm}{\em Let $F$ be a bounded subset of $\R^d$. For every $r\in (0,\infty)$, set $F_r := \{x\in\Z^d~:~r^{-1}x\in F\}$. Assume that $\frac{|F_r|}{r^d}$ converges to a nonzero limit when $r$ tends to infinity.

Then, one has $\lim_{r\to\infty} \frac{|\{x\in F_r~:~\gcd(x_1,\dots,x_d)=1\}|}{|F_r|}=1/\zeta(d)$.}
\tag{A}
\end{equation}

The study of \defini{coprime} vectors of $\Z^d$, i.e.~of the vectors $x$ that satisfy $\gcd(x_1,\dots,x_d)=1$, can be performed for its own sake. It may also be motivated by the reducibility of fractions \begin{small}(the probability that a random fraction is irreducible is $1/\zeta(2)=\frac{6}{\pi^2}$)\end{small} for $d=2$ or by the visibility problem for arbitrary $d$. If $x$ and $y$ denote two distinct points of $\Z^d$, one says that $x$ is \defini{visible} from $y$ if the line segment $[x,y]$ intersects $\Z^d$ only at $x$ and $y$. This condition is equivalent to $x-y$ having a \textsc{gcd} equal to 1. The set of visible points has been studied in various ways: see e.g. \cite{bcz2000, baakehuck, baakemoodypleasants, cff, garet, ferraguti, generalines, herzogpatterns, pleasantshuck, vardi}.

A classical corollary of (\ref{eq:classical}) is the following stronger statement.
Recall that the \defini{zeta distribution} of parameter $s>1$ is the probability distribution on $\N^\star$ giving weight $\frac{n^{-s}}{\zeta(s)}$ to each $n\in \N^\star$.

\begin{equation}
\label{eq:gcdfolklore}
\parbox{10.5cm}{\em Let $d\geq 2$. Let $F$ be a bounded subset of $\R^d$. For every $r\in (0,\infty)$, set $F_r := \{x\in\Z^d~:~r^{-1}x\in F\}$. Assume that $\frac{|F_r|}{r^d}$ converges to a nonzero limit when $r$ tends to infinity. Let $Y_r$ denote a uniform element in $F_r$.

Then, $\gcd(Y_r)$ converges in distribution to a zeta distribution of parameter $d$, as $r$ goes to infinity.}
\tag{B}
\end{equation}

\vspace{0.2cm}

\noindent In this paper, we are interested in the following two informal questions.

\vspace{0.2cm}

\begin{equation}
\label{eq:qfaible}
\parbox{10.5cm}{\em Colour an element of $\Z^d$ white if its coordinates are coprime and black otherwise. What does this colouring look like when seen from a ``uniformly chosen'' point of $\Z^d$?}
\tag{Q1}
\end{equation}

\begin{equation}
\label{eq:qfort}
\parbox{10.5cm}{\em Label every element of $\Z^d$ by its \textsc{gcd}. What do the labels look like around a ``uniform'' point of $\Z^d$?}
\tag{Q2}
\end{equation}

Even though (\ref{eq:qfort}) is stronger than (\ref{eq:qfaible}), it is worthwhile to study both questions independently. Indeed, this leads to two strategies of different nature, and the intermediate proposition related to (\ref{eq:qfort}) is no stronger than that of (\ref{eq:qfaible}) --- see Proposition~\ref{prop:coprime}, Proposition~\ref{prop:gcd}, and Fact~\ref{fact:cex}.

\newcommand{\cop}{\mathsf{cop}}
\newcommand{\gcdcol}{\mathsf{gcd}}

Let us now make sense of these two questions.
Let $X$ be a Polish space. Denote by $\Omega_X$ the space $X^{\Z^d}$ endowed with the product topology, which is also Polish. Let $c$ be an element of $\Omega_X$.
In (\ref{eq:qfaible}), we will be interested in $X=\{0,1\}$ and $c=\cop:=\mathds{1}_{\text{coprime}}$. In (\ref{eq:qfort}), we will be interested in $X=\N$ and $c=\gcdcol:x\mapsto \textsc{gcd}(x_1,\dots,x_d)$.

If $F$ is a nonempty finite subset of $\Z^d$, one defines the probability measure $\mu_{F,c}$ as follows --- it represents $c$ seen from a uniform point in $F$. For any ${y}\in \Z^d$ and $\omega\in \Omega_X$, define $\tau_{{y}} \omega \in \Omega_X$ by:
$$\forall x\in \Z^d,~(\tau_y\omega)_x=\omega_{x-y}.$$
Let $Y$ be a uniformly chosen element of $F$. We denote by $\mu_{F,c}$ the distribution of $\tau_{-{Y}}c$.

We want to describe the limit of $\mu_{F_n,c}$ when $c$ is a map of interest and $(F_n)$ is a reasonable sequence of finite subsets of $\Z^d$, such as $(\llbracket 1,N \rrbracket^d)$, $(\llbracket -N,N \rrbracket^d)$, or $(\{x\in \Z^d: \|x\|_2\leq n\})$. (\footnote{This way of proceeding is closely related to local convergence and local weak convergence \begin{small}(also called Benjamini--Schramm convergence)\end{small}: see \cite{babai91, bslimits, diestelleader}.}) Limits are taken in the following sense: the space of probability measures on $\Omega_X$ is endowed with the weak topology, namely the smallest topology such that for every bounded continuous map $f:\Omega_X\to \R$, the map $\mu\mapsto \int_{\Omega_X} f(\omega) d\mu(\omega)$ is continuous.

\vspace{0.2cm}

\begin{small}
When $X$ is discrete, $\mu_n$ converges to $\mu$ if and only if for every cylindrical event $A$, one has $\mu_n[A]\xrightarrow[n\to \infty]{} \mu[A]$. Recall that a \defini{cylindrical event} is an event of the form $\{\omega\in \Omega~:~\omega_{|F}\in A\}$, where $F$ is a finite subset of $\Z^d$, $\omega_{|F}$ stands for the restriction of $\omega$ to $F$, and $A$ is some measurable subset of $X^F$.
When $X$ is compact, the space of probability measures on $X$ is compact for the weak topology.
\end{small}

\begin{thm}
\label{thm:gcd}

Let $d\geq 2$. Let $F$ be a bounded convex subset of $\R^d$ with nonempty interior. For every $r\in (0,\infty)$, set $F_r := \{x\in\Z^d~:~r^{-1}x\in F\}$.

Then, $\mu_{F_r,\gcdcol}$ converges to some explicit probability measure $\mu_{\infty,\gcdcol}$ when $r$ goes to infinity.
\end{thm}

The probability measure $\mu_{\infty,\gcdcol}$ is defined on page \pageref{def:gcdinfty}. It does not depend on $F$ but only on $d$. Theorem~\ref{thm:gcd} answers (\ref{eq:qfort}), hence (\ref{eq:qfaible}).

\vspace{0.3cm}

The remaining of the paper is organised as follows. Section~\ref{sec:main} investigates (\ref{eq:qfaible}) and (\ref{eq:qfort}). This section contains the proof of Theorem~\ref{thm:gcd}, but also complementary results such as Proposition~\ref{prop:stochordering}. The use of relevant arithmetic/probabilistic notions makes the core of the problem very apparent: see $\footnotesize \textdbend$ on pages \pageref{subtil1} and \pageref{subtil2}. Section~\ref{sec:further}  provides several generalisations of these results.

\section{Answers to (\ref{eq:qfaible}) and (\ref{eq:qfort})}
\label{sec:main}

In Section~\ref{subsec:coprime}, we study (\ref{eq:qfaible}). In Section~\ref{subsec:gcd}, we investigate (\ref{eq:qfort}). Finally, in Section~\ref{subsec:rem}, we make several comments on Propositions~\ref{prop:coprime} and \ref{prop:gcd}, and we explore the percolative properties of $\mu_{\infty,\cop}$.

\subsection{The random coprime colouring}
\label{subsec:coprime}

In this section, we answer (\ref{eq:qfaible}) by proving Theorem~\ref{thm:coprime}. Even though Theorem~\ref{thm:coprime} is a particular case of Theorem~\ref{thm:gcd}, intermediate propositions of this section cannot be deduced from the content of Section~\ref{subsec:gcd}. See Fact~\ref{fact:cex}.

Let us define $\mu_{\infty,\cop}$, which will turn out to be the limit of $\mu_{F_n,\cop}$ when $(F_n)$ is a reasonable sequence of finite subsets of $\Z^d$. For every prime $p$, let $\mathcal{W}_p$ denote a uniformly chosen coset of $p\Z^d$ in $\Z^d$, i.e.~one of the $p^d$ sets of the form $x+p\Z^d$. Do all these choices independently. The distribution of $\bigcup_p \mathcal{W}_p$ is denoted by $\mu_{\infty,\cop}$. One can see $\mu_{\infty,\cop}$ as a probability measure on $\Omega_{\{0,1\}}$ by identifying a set with its indicator function.

\begin{thm}
\label{thm:coprime}
Let $d\geq 1$. Let $F$ be a bounded convex subset of $\R^d$ with nonempty interior. For every $r\in (0,\infty)$, set $F_r := \{x\in\Z^d~:~r^{-1}x\in F\}$.

Then, $\mu_{F_r,\cop}$ converges to $\mu_{\infty,\cop}$ when $r$ goes to infinity.
\end{thm}

\begin{rem}
The convergence of $\mu_{F_n,\cop}$ to $\mu_{\infty,\cop}$ for \emph{some sequences $(F_n)$ of balls} was conjectured by Vardi and obtained by Pleasants and Huck: see Conjecture~1 in \cite{vardi} and Theorem 1 in \cite{pleasantshuck}. Their vocabulary and techniques are very different from those of the present paper, and our approach further yields Propositions~\ref{prop:coprime}, \ref{prop:stochordering}, and \ref{prop:gcd}. See also the two last paragraphs of Section~\ref{subsec:rem}.
\end{rem}

We will actually prove Proposition~\ref{prop:coprime} instead of Theorem~\ref{thm:coprime}. Say that a sequence $(F_n)$ of finite nonempty subsets of $\Z^d$ is a \defini{F\o lner sequence} if for every ${y}\in \Z^d$, one has $|F_n \Delta (F_n+{y})|=o(|F_n|)$. \begin{small}It suffices to check this condition for a family of ${y}$'s generating $\Z^d$ as a group. This condition is also equivalent to $\lim_n \frac{|\partial F_n|}{|F_n|}=0$, where the \defini{boundary} $\partial F$ of $F\subset \Z^d$ is the set of the elements of $\Z^d\backslash F$ that are adjacent to an element of $F$ for the usual (hypercubic) graph structure of $\Z^d$.\end{small}

\begin{prop}
\label{prop:coprime}
Let $d\geq 1$ and let $(F_n)$ be a F\o lner sequence of $\Z^d$. Assume that $\mu_{F_n}(\{\omega~:~(0,\dots,0)\in\omega\})$ converges to $1/\zeta(d)$.

Then, $\mu_{F_n,\cop}$ converges to $\mu_{\infty,\cop}$.
\end{prop}

Theorem~\ref{thm:coprime} immediately follows from (\ref{eq:classical}) and Proposition~\ref{prop:coprime}, which it thus suffices to prove. In Section~\ref{subsec:rem}, we will see that none of the F\o lner condition and the $1/\zeta(d)$-assumption can be removed from Proposition~\ref{prop:coprime}. This proposition will itself be deduced from the following stochastic domination result. If $\mu$ and $\nu$ denote two probability measures on $\Omega_{\{0,1\}}$, we say that $\mu$ is \defini{stochastically dominated} by $\nu$ if there is a coupling $(\mathcal{W},\mathcal{W}')$ of $(\mu,\nu)$ such that $\mathcal{W}\subset \mathcal{W}'$ almost surely.

\begin{prop}\label{prop:stochordering}
Let $d\geq 1$ and let $(F_n)$ be a F\o lner sequence of $\Z^d$. Assume that $\mu_{F_n,\cop}$ converges to some probability measure $\mu$.

Then, $\mu$ is stochastically dominated by $\mu_{\infty,\cop}$.
\end{prop}

The goal of the remaining of this section is to prove Proposition~\ref{prop:coprime} and Proposition~\ref{prop:stochordering}.

\begin{lem}
\label{lem:chinese-folner}
Let $d\geq 1$, let $N$ be a positive integer, and let $\pi:\Z^d\to (\Z/N\Z)^d$ denote reduction modulo $N$. Let $(F_n)$ be a F\o lner sequence of $\Z^d$, and let $Y_n$ denote a uniformly chosen element of $F_n$.

Then, $\pi(Y_n)$ converges in distribution to the uniform measure on $(\Z/N\Z)^d$.
\end{lem}

\vspace{0.3cm}

\begin{proof}
Partition $\Z^d$ into boxes of the form $\prod_{i=1}^d\llbracket Nx_i, N(x_i+1)-1\rrbracket$. For every $n$, say that an element $x$ of $F_n$ is $n$-good if the box $B$ containing it satisfies $B\subset F_n$. Let $Y_n$ denote a uniformly chosen element of $F_n$. Because $(F_n)$ is a F\o lner sequence, the probability that $Y_n$ is $n$-good converges to 1 as $n$ goes to infinity. But if $Y'_n$ denotes a uniformly chosen \emph{$n$-good} element of $F_n$, then $\pi(Y'_n)$ is precisely uniform in $(\Z/N\Z)^d$: one may generate $Y'_n$ by picking independently a uniform element ${Y}$ in $\llbracket 0, N-1\rrbracket^d$ and a uniform element $Y_n''$ in $\{x\in N\Z^d : x\text{ is }n\text{-good}\}$, and then writing $Y'_n=Y_n''+{Y}$. The result follows.
\end{proof}

\newcommand{\pcop}{\mathsf{pcop}}

We will need the forthcoming Lemma~\ref{lem:coprime}, which requires the following definitions of $\pcop$ and $\mu_{\infty,\pcop}$.
Let $X=\{0,1\}^\Prime$. The map $\pcop : \Z^d \to X$ is given by $\pcop(x)_p=\mathds{1}_{x\notin p\Z^d}$. Let us now define $\mu_{\infty,\pcop}$. For every $p\in \Prime$, pick a uniform coset $\mathcal{W}_p$ of $p\Z^d$ among the $p^d$ possible ones. Do this independently for every $p$. One defines the random element $\mathcal{W}'$ of $\Omega_X$ via $\mathcal{W}'(x)_p:=\mathds{1}_{x\notin \mathcal{W}_p}$. 
We will denote by $\mu_{\infty,\pcop}$ the distribution of the random variable $\mathcal{W}'$. Notice that the random element $\mathcal{W}$ of $\Omega_{\{0,1\}}$ defined by $\mathcal{W}(x):=\min_p \mathcal{W}'(x)_p$ has distribution $\mu_{\infty,\cop}$.

\begin{lem}
\label{lem:coprime}
Let $d\geq 1$ and let $(F_n)$ be a F\o lner sequence of $\Z^d$.

Then, $\mu_{F_n,\pcop}$ converges to $\mu_{\infty,\pcop}$.
\end{lem}

Lemma~\ref{lem:coprime} can be deduced from Lemma~\ref{lem:chinese-folner}. We will instead deduce it from Lemma~\ref{lem:zhat}: see Remark~\ref{rem:zhattoeasier}. Our task is now to get Proposition~\ref{prop:stochordering} from Lemma~\ref{lem:coprime}, and Proposition~\ref{prop:coprime} from Proposition~\ref{prop:stochordering}. But before doing so, I would like to explain why Proposition~\ref{prop:coprime} does \emph{not} follow directly from Lemma~\ref{lem:coprime}.

\vspace{0.24cm}

\textdbend\ To\label{subtil1} deduce directly Proposition~\ref{prop:coprime} from Lemma~\ref{lem:coprime}, one would need the continuity of the map $f : \Omega_{\{0,1\}^\Prime}\to \Omega_{\{0,1\}}$ defined by $f(\omega)_x=\min_p \omega(x)_p$. Indeed, $\cop = f \circ \pcop$. But this continuity, equivalent to that of the $\min$ map defined from $\{0,1\}^\Prime$ to $\{0,1\}$, does \emph{not} hold. However, the $\min$ map is \emph{upper semicontinuous}, which permits the proof of Proposition~\ref{prop:stochordering}.
Notice that there can be no way to do ``as if'' $f$ was continuous, as the F\o lner assumption of Lemmas~\ref{lem:chinese-folner}, \ref{lem:coprime} and \ref{lem:zhat} does \emph{not} imply that $\mu_{F_n,\cop}$ converges to $\mu_{\infty,\cop}$ --- see Remark~\ref{rem:folner}.

\vspace{0.54cm}

\begin{proofdomisto} Let us make the assumptions of Proposition~\ref{prop:stochordering}.
For $Y_n$ uniformly chosen in $F_n$, let us consider $(\tau_{-Y_n}\cop,\tau_{-Y_n}\pcop)\in \Omega_{\{0,1\}}\times \Omega_{\{0,1\}^\Prime}\cong \Omega_{\{0,1\}\times \{0,1\}^\Prime}$. As $\{0,1\}\times \{0,1\}^\Prime$ is compact, up to passing to a subsequence, we may assume that the distribution of $(\tau_{-Y_n}\cop,\tau_{-Y_n}\pcop)$ converges to some probability measure $\rho$ on $\Omega_{\{0,1\}}\times \Omega_{\{0,1\}^\Prime}$. Notice that for every $x\in \Z^d$ and $p\in \Prime$, one has
$$
\pcop(x)_p=0\implies \cop(x)=0.
$$
Besides, for every $x$ and $p$ \begin{small}(which range over countable sets)\end{small}, the set $$\{(\omega,\omega')~:~\omega(x)\leq \omega'(x)_p\}$$ is open and closed inside $\Omega_{\{0,1\}}\times \Omega_{\{0,1\}^\Prime}$.
As a result, for $\rho$-almost every $(\omega,\omega')$, for every $x\in \Z^d$, one has $\omega(x)\leq \min_p \omega'(x)_p$. But recall that if $(\mathcal{W},\mathcal{W}')$ denotes a random variable with distribution $\rho$, then $\mathcal{W}$ has distribution $\mu$, $\mathcal{W}'$ has distribution $\mu_{\infty,\pcop}$ --- by Lemma~\ref{lem:coprime} ---, and thus $\min_p \mathcal{W}'(x,p)$ has distribution $\mu_{\infty,\cop}$. Proposition~\ref{prop:stochordering} follows.
\end{proofdomisto}

\vspace{0.44cm}

\begin{proofpropcoprime} Let us make the assumptions of Proposition~\ref{prop:coprime}. Up to taking a subsequence of $(F_n)$, we may assume that $\mu_{F_n,\cop}$ converges to some $\mu$. We want to prove that $\mu=\mu_{\infty,\cop}$. By Proposition~\ref{prop:stochordering}, there is a monotone coupling of $(\mu,\mu_{\infty,\cop})$, i.e.~some coupling $\rho$ of $(\mu,\mu_{\infty,\cop})$ such that for $\rho$-almost every $(\omega,\omega_\infty)$, one has $\forall x,~\omega(x)\leq \omega_\infty(x)$. But by the $1/\zeta(d)$-assumption, one has
$$
\mu(\{\omega:\omega(\vec{0})=1\})=1/\zeta(d)=\prod_{p\in \Prime}(1-p^{-d})=\mu_{\infty,\cop}(\{\omega:\omega(\vec{0})=1\}).
$$
As a result, for $\rho$-almost every $(\omega,\omega_\infty)$, one has $\omega(\vec{0})=\omega_\infty(\vec{0})$. The same argument applies for any $x\in\Z^d$ instead of $\vec{0}$. Indeed, the probability measure $\mu_{\infty,\cop}$ is translation invariant by construction, and $\mu$ is translation invariant because $(F_n)$ is a F\o lner sequence. Thus, $\mu$ is equal to $\mu_{\infty,\cop}$ and Proposition~\ref{prop:coprime} is proved.

\end{proofpropcoprime}

\subsection{The random $\gcd$ labelling}
\label{subsec:gcd}

This section is devoted to proving Theorem~\ref{thm:gcd}.

\vspace{0.2cm}

Let us first define $\mu_{\infty,\gcdcol}$.\label{def:gcdinfty} For every prime $p$, perform the following random choices:
\begin{itemize}
\item set $\mathcal{W}_0^p:=\Z^d$,
\item conditionally on $(\mathcal{W}_0^p,\dots,\mathcal{W}_{n-1}^p)$, pick a uniform coset $\mathcal{W}_n^p$ of $p^n\Z^d$ among those lying inside $\mathcal{W}_{n-1}^p$.
\end{itemize}
Do this independently for every $p$. We set the random $p$-adic valuation of a vertex $x$ in $\Z^d$ to be $V_p(x):=\sup\{n\in \N~:~x\in \mathcal{W}_n^p\}\in \llbracket 0,\infty\rrbracket$. We define the random \textsc{gcd} profile to be the random map $x\mapsto \prod_{p\in\Prime}p^{V_p(x)}$. \begin{small}(This occurs almost surely nowhere, but one should set $\prod_{p\in\Prime}p^{V_p(x)}$ to be 0 whenever $\forall p,~V_p(x)=\infty$.)\end{small} The distribution of the random \textsc{gcd} profile is denoted by $\mu_{\infty,\gcdcol}$. It is a priori a probability distribution on $\Omega_{\llbracket 0,\infty\rrbracket}$. By the Borel--Cantelli Lemma, for every $d\geq 2$, it is also a probability distribution on $\Omega_\N$, and it is as such that it will be considered from now on.

We will actually prove the following proposition.

\begin{prop}
\label{prop:gcd}
Let $d\geq 1$ and let $(F_n)$ be a F\o lner sequence of $\Z^d$. For every $n$, let $Y_n$ denote a uniformly chosen element of $F_n$. Assume that the distribution of $\gcd(Y_n)$ is tight \begin{small}(which implies $d\geq 2$)\end{small}.

Then, $\mu_{F_n,\gcdcol}$ converges to $\mu_{\infty,\gcdcol}$.
\end{prop}

Theorem~\ref{thm:gcd} immediately follows from (\ref{eq:gcdfolklore}) and Proposition~\ref{prop:gcd}, which it thus suffices to prove. The forthcoming Lemma~\ref{lem:zhat} uses the notion of profinite numbers. Define $\Zhat$ as follows: $$\hat{\Z}=\varprojlim_n \Z/n\Z=\left\{x\in \prod_{n\geq 1}\Z/n\Z :\forall (m,n),~m|n\implies \pi_{m,n}(x_n)=x_m\right\},$$ where $\pi_{m,n}$ denotes the morphism of reduction modulo $m$ from $\Z/n\Z$ to $\Z/m\Z$. One can see $\Z$ as a dense subgroup of $\hat{\Z}$ via the injection
$$\Phi : N\mapsto (n\mapsto N+n\Z).$$
Elements of $\Zhat$ are called \defini{profinite numbers}.
The product topology on $\prod_{n\geq 1}\Z/n\Z$ induces on $\hat{\Z}$ and on $\Z\subset \hat{\Z}$ the so-called \defini{profinite topology}. It makes of $(\Zhat,+)$ a metrisable compact topological group.

\newcommand{\prof}{\mathfrak}
\newcommand{\profcol}{\mathsf{prof}}

Recall that when $p$ is a prime, the set of \defini{$p$-adic integers} is defined by $\Z_p=\varprojlim_n \Z/p^n\Z$. The Chinese Remainder Theorem implies that $\Zhat$ and $\prod_p \Z_p$ are isomorphic as topological groups (and even as rings), an isomorphism being given by $$\Zhat \ni \prof{n}\mapsto \left(n\mapsto \prof{n}_{p^n}\right)_{p\in \Prime}\in \prod_p \Z_p.$$

Let $\profcol:\Z^d\to\Zhat^d$ be defined by $\profcol(x)=(\Phi(x_1),\dots,\Phi(x_d))\in \Zhat^d$. In order to define a suitable probability measure $\mu_{\infty,\profcol}$, notice that there is a unique Haar measure on the compact group $\Zhat\cong \prod_p \Z_p$, which corresponds to the product of Haar measures on $\Z_p$. A Haar distributed element of $\Z_p$ is a random $p$-adic integer $(X_n)_{n\geq0}$ such that for every $n\geq 1$, conditionally on $X_{n-1}$, the element $X_n$ is uniform among the $p$ elements of $\Z/p^n\Z$ that reduce to $X_{n-1}$ modulo $p^{n-1}$.

Pick a Haar distributed element $Y$ of $\Zhat^d$, i.e.~$d$ independent Haar distributed elements of $\Zhat$. We denote by $\mu_{\infty,\profcol}$ the distribution of the following random element of $\Omega_{\Zhat^d}$:
$$
x\mapsto Y+\profcol(x).
$$

\begin{lem}
\label{lem:zhat}
Let $d\geq 1$ and let $(F_n)$ be a F\o lner sequence of $\Z^d$.

Then, $\mu_{F_n,\profcol}$ converges to $\mu_{\infty,\profcol}$.
\end{lem}

\begin{proof}
As $\Zhat$ is compact, up to passing to a subsequence, we may assume that $\mu_{F_n,\profcol}$ converges to some probability measure $\mu$.

Let $Y_n$ denote a uniformly chosen element of $F_n$, and set $X_n:=\Phi(Y_n)$. By Lemma~\ref{lem:chinese-folner}, for every $N$, the $(N!)$-component of $X_n$ converges in distribution to a uniform element of $\Z/N!~\Z$. As a result, $X_n$ converges in distribution to a Haar distributed element of $\Zhat$. Therefore, if $\mathfrak{s}$ is $\mu$-distributed, then $\mathfrak{s}(0)$ is Haar distributed on $\Zhat$.

It thus suffices to prove that $\mu$ gives probability 1 to the following set: $\{\sigma\in \Omega_{\Zhat}~:~\forall x\in\Z^d,~\sigma(x)=\sigma(\vec{0})+\profcol(x)\}$. But this is clear as it is closed and has probability 1 for every $\mu_{F_n,\profcol}$.
\end{proof}

\begin{rem}\label{rem:zhattoeasier} As the map $f:\Zhat \to \{0,1\}^\Prime$ defined by $f(\mathfrak{n})_p:=\mathds{1}_{\mathfrak{n}_p\not=0}$ is continuous, Lemma~\ref{lem:coprime} follows directly from Lemma~\ref{lem:zhat}.
\end{rem}

\newcommand{\SuperN}{\mathfrak{N}}
\newcommand{\supercol}{\mathsf{super}}

As we are interested in the \textsc{gcd} of random elements of $\Z$, it is useful to recall what the \textsc{gcd} of a profinite number is. While the \textsc{gcd} of an element of $\Z$ is a natural number (possibly 0), the \textsc{gcd} of a profinite number will be a \emph{supernatural} number. The set of \defini{supernatural numbers} is $\SuperN:=\llbracket 0,\infty\rrbracket^\Prime$. The set $\llbracket 0,\infty\rrbracket$ is endowed with the usual topology \begin{small}(a set $U$ is open if and only if for $n$ large enough, $n\in U \iff \infty \in U$)\end{small}, and $\SuperN=\llbracket 0,\infty\rrbracket^\Prime$ is endowed with the product of these topologies. One can see any nonnegative integer $n\in \N$ as a supernatural number via the following injection:
$$
\Psi : n \mapsto (p\mapsto p\text{-adic valuation of }n).
$$
The $p$-adic valuation of a profinite number $\mathfrak{n}$ is
$$
v_p(\mathfrak{n}):=\sup \{k\in \N~:~\mathfrak{n}_{p^k}=0\} \in \llbracket 0,\infty\rrbracket.
$$
The \textsc{gcd} of $(\mathfrak{n}_1,\dots,\mathfrak{n}_d)$ is the following supernatural number:
$$
p\mapsto \min(v_p(\mathfrak{n}_1),\dots,v_p(\mathfrak{n}_d)).
$$

Let $\supercol : \Z^d \to \SuperN$ be defined by $x\mapsto \Psi(\gcd(x))$.
Given a Haar distributed element $Y$ of $\Zhat^d$, denote by $\mu_{\infty,\supercol}$ the distribution of the following random element of $\Omega_{\SuperN}$:
$$
x\mapsto \gcd(Y+\profcol(x)).
$$

\begin{lem}
\label{lem:super}
Let $d\geq 1$ and let $(F_n)$ be a F\o lner sequence of $\Z^d$.

Then, $\mu_{F_n,\supercol}$ converges to $\mu_{\infty,\supercol}$.
\end{lem}

\begin{proof}
As the \textsc{gcd}-map is continuous from $\Zhat$ to $\SuperN$, this follows directly from Lemma~\ref{lem:zhat}.
\end{proof}

Notice that the conclusion of Lemma~\ref{lem:super} is quite close to the one we are looking for. Assume further that $d\geq 2$. Then, as $\sum_{p\in\Prime} p^{-d}<\infty$, the Borel--Cantelli Lemma yields that if $\mathbf{X}$ is $\mu_{\infty, \supercol}$-distributed, then almost surely, every $x\in \Z^d$ satisfies $\mathbf{X}(x)\in \Psi(\N^\star)$. Recall that $\Psi(\N^\star)$ consists in finitely supported elements of $\SuperN=\llbracket 0,\infty\rrbracket^\Prime$ all the values of which are finite. If $\mathbf{n}\in \Psi(\N^\star)$, then $\mathbf{n}=\Psi\left(\prod_{p\in \Prime} p^{\mathbf{n}_p}\right)$ and we write $\chi(\mathbf{n})=\prod_{p\in \Prime} p^{\mathbf{n}_p}\in \N^\star$. Therefore, the random variable $(\chi(\mathbf{X}(x)))_{x\in\Z^d}$ has distribution $\mu_{\infty,\gcdcol}$.

\textdbend\  If\label{subtil2} $\chi$ was continuous and defined everywhere on $\SuperN$, one could directly deduce from Lemma~\ref{lem:super} that $\mu_{F_n,\gcdcol}$ converges to $\mu_{\infty,\gcdcol}$; but the actual properties of $\chi$ do not allow us to perform this derivation, even in an indirect manner. Indeed, $(F_n)$ being F\o lner and $d$ being at least 2 do \emph{not} suffice to guarantee this convergence: these conditions do not even suffice to guarantee convergence of $\mu_{F_n,\cop}$ to $\mu_{\infty,\cop}$. See Remark~\ref{rem:folner}. We need the tightness assumption and this goes through the following easy lemma.

\vspace{0.2cm}

Denote by $f_\sharp \rho$ the pushforward by $f$ of a given measure $\rho$.

\begin{lem}
\label{lem:topolike}
Let $X$ and $Y$ be Polish spaces. Let $f:{X}\to {Y}$ be a continuous injective map that maps every Borel subset of $X$ to a Borel subset of $Y$. Let $(\mu_n)_{n\leq \infty}$ denote a sequence of probability measures on ${X}$. For $n\leq \infty$, set $\nu_n := f_\sharp \mu_n$. Assume that $\nu_n$ converges to $\nu_\infty$ and that $\mu_n$ converges to some probability measure $\mu$.

Then $\mu$ and $\mu_\infty$ are equal.
\end{lem}

\begin{proof}
By continuity of $f$, the sequence $\nu_n=f_\sharp \mu_n$ converges to $f_\sharp \mu$. As this sequence also converges to $f_\sharp \mu_\infty$, we have $f_\sharp \mu=f_\sharp \mu_\infty$. By injectivity of $f$, we have $\mu(A)=f_\sharp \mu (f(A))=f_\sharp \mu_\infty (f(A))=\mu_\infty (f(A))$, where the assumptions indeed guarantee that $f(A)$ is Borel.
\end{proof}

Notice that in Lemma~\ref{lem:topolike}, if one does not assume that $\mu_n$ converges to some probability measure, then one cannot deduce that $\mu_n$ converges to $\mu_\infty$. A counterexample goes as follows. Take ${X}=(0,1]$, ${Y}=\R/\Z$, and set $f:{X}\to {Y}$ to be reduction modulo 1. For $n\in \N^\star$, set $\mu_n=\delta_{1/n}$, and let $\mu_\infty=\delta_1$.

\vspace{0.38cm}

\begin{proofpropgcd}
Let $(F_n)$ and $(Y_n)$ be as in Proposition~\ref{prop:gcd}. By hypothesis, $\gcd(Y_n)$ is tight. Together with the assumption that $(F_n)$ is a F\o lner sequence, this implies that $d\geq 2$ and that for every $x\in \Z^d$, $\gcd(Y_n+x)$ is tight. Therefore, the ($\N$-indexed) sequence of random variables $(\gcd(Y_n+x))_{x\in \Z^d}$ is tight. Up to passing to a subsequence, we may thus assume that its distribution converges to some probability measure $\mu$ on $\Omega_\N$.

Let $X:=\Omega_{\N}$ and $Y:=\Omega_\SuperN$. For $n\in \N$, let $\mu_n:=\mu_{F_n,\gcdcol}$. Set $\mu_\infty:=\mu_{\infty,\gcdcol}$, which is a well-defined probability measure on $\Omega_{\N}$ because of the Borel--Cantelli Lemma \begin{small}($\sum_{p\in \Prime}p^{-d}<\infty$)\end{small}. Let $f:X\to Y$ be defined by $f(\sigma)=(\Psi(\sigma_x))_{x\in\Z^d}$. This map is injective, continuous, and maps Borel subsets of $X$ to Borel subsets of $Y$. By Lemma~\ref{lem:super}, $f_\sharp \mu_n$ converges to $f_\sharp \mu_\infty$. As $\mu_n$ converges to $\mu$, Lemma~\ref{lem:topolike} yields that $\mu_n$ converges to $\mu_\infty$, which is the desired result.
\end{proofpropgcd}

\subsection{Remarks}
\label{subsec:rem}

We insist that in none of our results, we ask for the sequence $(F_n)$ to be monotone, or for $\bigcup_n F_n$ to be equal to $\Z^d$. The fact that $|F_n|$ tends to infinity is a consequence of being F\o lner.

\vspace{0.2cm}

Even though the conclusion of Proposition~\ref{prop:gcd} implies that of Proposition~\ref{prop:coprime}, it is not the case that Proposition~\ref{prop:coprime} can be derived from Proposition~\ref{prop:gcd}. Indeed, one has the following fact.

\begin{fact}
\label{fact:cex}
For every $d\geq 1$, there is a F\o lner sequence $(F_n)$ of $\Z^d$ such that the following conditions hold:
\begin{itemize}
\item if $Y_n$ denotes a uniform element of $F_n$, the distribution of $\gcd(Y_n)$ is not tight,
\item the proportion of coprime vectors in $F_n$ converges to $1/\zeta(d)$.
\end{itemize}
\end{fact}

\begin{proof} This fact can be derived from the Chinese Remainder Theorem or from the proof of Theorem~\ref{thm:gcd}. We will proceed in the second manner.

If $d=1$, any F\o lner sequence satisfies automatically the desired properties. Let us thus assume that $d\geq 2$, so that $\sum_{p\in\Prime} p^{-d}<\infty$.
Let $\varepsilon>0$ and $K\in \N$. One has
$$
\mu_{\infty,\supercol}\left(\left\{\sigma~:~\forall p,~\sigma(0)_p=0\right\}\right)=\prod_{p\in \Prime}(1-p^{-d})=1/\zeta(d).
$$
By translation invariance, one has
$$
\int \frac{|\{x\in \llbracket 1,N\rrbracket^d~:~\forall p,~\sigma(x)_p=0\}|}{N^d}\text{d}\mu_{\infty,\supercol}(\sigma)=1/\zeta(d).
$$
Either by ergodicity\footnote{Later in this section, we will see that $\mu_{\infty,\supercol}$ is ergodic, i.e.~that any translation invariant event has probability 0 or 1. One can then use an ergodic theorem such as Theorem~1.1 in \cite{lindenstrauss}.} or because of Proposition~\ref{prop:stochordering}, one can pick an arbitrarily large $N$ such that this proportion has a positive probability to be $\varepsilon$-close to $1/\zeta(d)$. Fixing such an $N$, one can thus find a deterministic subset $U$ of $\llbracket 1,N\rrbracket^d$ such that: $$\left|\frac{|U|}{N^d}-\frac{1}{\zeta(d)}\right|\leq \varepsilon~~\text{and}~~\mu_{\infty,\supercol}\left(\left\{\sigma~:~\forall x\in U,~\forall p\in\Prime,~\sigma(x)_p=0\right\}\right)>0.$$

Let $\mathfrak{s}$ be a random variable with distribution $\mu_{\infty,\supercol}$. As $d\geq 2$, one can pick some $P\geq \max(K,N)$ such that with positive probability, the following two conditions hold:
\begin{enumerate}
\item $\forall x\in U,~\forall p\in \Prime,~\mathfrak{s}(x)_p=0$,
\item $\forall x\in \llbracket 1,N\rrbracket^d,~\forall p\in \Prime,~(p\geq P\implies \mathfrak{s}(x)_p=0)$.
\end{enumerate}
Recall that $\mu_{\infty,\supercol}$ is defined in terms of the cosets $\mathcal{W}^n_p$. Let $\mathbf{p}$ denote an injective map from $\llbracket 1,N\rrbracket^d$ to $\{p\in \Prime : p\geq P\}$. We will use this injection to modify $\mathfrak{s}$ into a new random variable $\mathfrak{s}'$. For every $x\in \llbracket 1,N\rrbracket^d \backslash U$, resample $\left(\mathcal{W}_n^{\mathbf{p(x)}}\right)_n$ and condition $\mathcal{W}_1^{\mathbf{p(x)}}$ to contain $x$: do this independently of $\mathfrak{s}$ and independently for every $x\in \llbracket 1,N\rrbracket^d \backslash U$. The distribution of the ensuing random variable $\mathfrak{s}'$ has a density relative to $\mu_{\infty,\supercol}$.
As $\mathfrak{s}'$ almost surely satisfies the following conditions:
\begin{enumerate}
\item $\forall x\in U,~\prod_p p^{\mathfrak{s}'(x)_p}=1$,
\item $\forall x\in \llbracket 1,N\rrbracket^d\backslash U,~ \prod_p p^{\mathfrak{s}'(x)_p}\geq K$,
\end{enumerate}
the corresponding conditions are satisfied by $\mathfrak{s}$ with positive probability. Denote the corresponding event by $E$.

By Theorem~\ref{thm:gcd}, the proportion of $x\in \llbracket 1,n\rrbracket^d$ such that $\tau_{-x}\supercol$ satisfies $E$ converges to a positive number, hence is positive for $n$ large enough. In particular, there is some $y\in \Z^d$ such that $\tau_{-y}\supercol$ satisfies $E$. We say that any such $y$ is a $(K,N,\varepsilon)$-counterexample. For every $n$, pick some $y^{n}$ that is an $(n,m_n,1/n)$-counterexample for some $m_n\geq n$: the sequence $F_n=\prod_{i=1}^d\llbracket y_i^{n} +1, y_i^{n}+m_n\rrbracket$ satisfies the desired properties.
\end{proof}

\begin{rem}
\label{rem:folner}
Likewise, the $1/\zeta(d)$-condition cannot be removed from Proposition~\ref{prop:coprime}: it is well-known that the Chinese Remainder Theorem implies that there are arbitrarily large boxes $F_n={x_n} + \llbracket 0, N\rrbracket^d$ in $\Z^d$ devoid of coprime points. This also follows from Theorem~\ref{thm:coprime}, as $\mu_{\infty,\cop}$ has a positive probability to colour black the whole box $\llbracket 0, N\rrbracket^d$ --- use one prime per point of the box. Such a sequence $(F_n)$ is F\o lner and yet $\mu_{F_n,\cop}$ converges $\delta_\mathbf{0}$.
The F\o lner condition cannot be removed from Proposition~\ref{prop:coprime} either.
\end{rem}

\vspace{0.2cm}

All the results in this paper concerning visibility extend readily to lattices in $\R^d$, as any such lattice may be mapped to $\Z^d$ by a linear automorphism of $\R^d$, which preserves visibility.

\vspace{0.2cm}

Notice that $\mu_{\infty,\cop}$ and $\mu_{\infty,\gcdcol}$ are not only translation invariant but also $\mathsf{GL}_d(\Z)$-invariant. Even though $\cop$ and $\gcdcol$ are indeed $\mathsf{GL}_d(\Z)$-invariant, the $\mathsf{GL}_d(\Z)$-invariance of these measures does \emph{not} follow from Theorems~\ref{thm:gcd} and \ref{thm:coprime}. One way to understand why goes as follows. Every orbit of the multiplicative group $G$ generated by $\begin{pmatrix} 
1 & 1 \\
0 &  1
\end{pmatrix}$ contains a unique point inside $A:=\{x\in \Z^2:0\leq x_1<|x_2|\}\cup (\Z\times\{0\})$. Consider the colouring of $\Z^2$ which is defined as a the chessboard colouring on $A$ \begin{small}(say white if the sum of coordinates is odd and black otherwise)\end{small}, and take its unique $G$-invariant extension. For $F_n=A~\cap~\llbracket -n,n\rrbracket^2$, this colouring seen from a uniform point in $F_n$ converges to the unbiased choice of one of the two chessboard colourings of the plane: this probability measure is \emph{not} $G$-invariant.

\vspace{0.2cm}

Finally, let us mention a few questions that can be asked only with the formalism of local limits, i.e.~with $\mu_{\infty,\cop}$ or $\mu_{\infty,\gcdcol}$ instead of simply one convergence result per cylindrical event as in \cite{pleasantshuck}. For every $c\in\{\cop,\gcdcol,\profcol,\supercol\}$, we can ask whether the measure $\mu_{\infty,c}$ is ergodic --- i.e.~if every translation invariant measurable subset has $\mu_{\infty,c}$-probability 0 or 1. These measures are indeed ergodic. For simplicity, we expose the argument for $\mu_{\infty,\cop}$. The Chinese Remainder Theorem guarantees the following fact: if $p_1,\dots,p_n$ denote the first $n$ primes, and if for every $i\leq n$, $C_i$ denotes a coset of $p_i \Z^d$ in $\Z^d$, then there is a translation mapping every $C_i$ to $p_i\Z^d$. As a result, whenever we consider only finitely many primes, choosing one coset per prime yields a \emph{deterministic} outcome up to translation. One concludes by noting that if a translation invariant probability measure on $\{0,1\}^{\Z^d\times \Prime}$ yields ergodic measures in projection to any $\{0,1\}^{\Z^d\times \{p_1,\dots,p_n\}}$, then the measure under study is itself ergodic. This is easily proved by martingale theory, and similar reasonings are classical in the study of profinite actions.

\vspace{0.2cm}

Considering $\Z^d$ to be endowed with its usual \begin{small}(hypercubic)\end{small} graph structure, one may also ask questions of percolation theory \cite{grimmettbook, lyonsperes}: how many infinite white (resp. black) connected components does the colouring $\mu_{\infty,\cop}$ yield? By ergodicity, these numbers have to be deterministic outside some event of probability zero. One can derive from Theorem 3.3 in \cite{vardi} that, for $d=2$ hence for any $d\geq 2$, there is at least one infinite white connected component almost surely. One can derive from Theorem 3.4 in \cite{vardi} that, for $d=2$, there is almost surely at most one infinite white connected component and no infinite black component.

\section{Further generalisations of Theorem~\ref{thm:gcd}}
\label{sec:further}

\subsection{Sampling along affine subspaces}

\newcommand{\gcdinfty}{\mathsf{gcd}_\infty}

Proposition~\ref{thm:gcdaffine} is the analog of Proposition~\ref{prop:gcd} for an observer picked uniformly in an affine subspace of $\Z^d$ rather than in the whole space $\Z^d$.

\vspace{0.2 cm}

Endow $\llbracket 0,\infty\rrbracket$ with its usual topology. Define $\gcdinfty \in \Omega_{\llbracket 0,\infty\rrbracket}$ by setting $\gcdinfty(x):=\gcd(x)$. 
Let $({\mu_n})_{n\leq\infty}$ denote a sequence of probability measures on $\Omega_{\llbracket 0,\infty\rrbracket}$. Let $V\subset \Z^d$ and denote by $\pi$ the projection from $\llbracket 0,\infty\rrbracket^{\Z^d}$ to $\llbracket 0,\infty\rrbracket^{V}$. We say that $\mu_n$ converges to $\mu_\infty$ \defini{on $V$} if $\pi_\sharp \mu_n$ --- the pushforward of $\mu_n$ by $\pi$ --- converges to $\pi_\sharp \mu_\infty$ when $n$ goes to infinity.

Given $\Gamma$ an infinite subgroup of $\Z^d$, one defines the probability measure $\mu_{\infty,\gcdinfty}^{\Gamma}$ by taking the definition of $\mu_{\infty,\gcdcol}$ but asking furthermore at every step that $\mathcal{W}_n^p$ intersects $\Gamma$. This corresponds to taking a Haar distributed element in the closure of $\Gamma$ in $\hat{\Z}^d$.

Finally, we say that a sequence $(F_n)$ of nonempty finite subsets of $\Gamma$ is a \defini{F\o lner sequence for $\Gamma$} if for every ${y}\in \Gamma$, one has $|F_n\Delta (F_n+{y})|=o(|F_n|)$.

\begin{prop}
\label{thm:gcdaffine}
Let $\Gamma$ denote a subgroup of $\Z^d$ of rank $k\geq 1$ which is maximal among subgroups of rank $k$. Let $(F_n)$ be a F\o lner sequence for $\Gamma$. Let $Y_n$ denote a uniform element of $F_n$.

Then, $\mu_{F_n,\gcdinfty}$ converges to $\mu_{\infty,\gcdinfty}^{\Gamma}$ on $\Z^d\backslash \Gamma$.
If furthermore $k=1$ or $\gcd(Y_n)$ is tight, then $\mu_{F_n,\gcdinfty}$ converges to $\mu_{\infty,\gcdinfty}^{\Gamma}$.
\end{prop}

\begin{rem}
\label{rem:gamma}
The view from a ``uniform point'' in an \emph{affine} subspace $\Gamma+{y}$ is just the view seen from a ``uniform point'' in $\Gamma$ shifted by $-{y}$. One may also notice that if one starts with a group $\Gamma$ that is not maximal given its rank, then it lies in a unique such group, which is the intersection of its \begin{small}(rational or real)\end{small} linear span with $\Z^d$ : denote it by $\tilde{\Gamma}$. It follows from the proof of Proposition~\ref{thm:gcdaffine} that the following holds. Let $(F_n)$ be a F\o lner sequence for $\Gamma$. Let $Y_n$ denote a uniform element of $F_n$. Assume that for every ${y}\in \tilde{\Gamma}$, $\gcd(Y_n+y)$ is tight. Then $\mu_{F_n,\gcdinfty}$ converges to $\mu_{\infty,\gcdinfty}^{\Gamma}$. Actually, it suffices to make the assumption for a system of representatives of ${y}\in \tilde{\Gamma}$ for the equivalence relation ``being equal modulo $\Gamma$'', i.e.~for finitely many ${y}$'s.
\end{rem}

\begin{coro}
\label{coro:gcdaffine}
Let $\Gamma$ denote a subgroup of $\Z^d$ of rank $k\geq 1$ which is maximal among subgroups of rank $k$. Let $F$ denote a bounded subset of the linear span of $\Gamma$. For every $r\in (0,\infty)$, set $F_r := \{x\in\Gamma~:~r^{-1}x\in F\}$. Assume that $\frac{|F_r|}{r^k}$ converges to a nonzero limit when $r$ tends to infinity and that $(F_n)$ is a F\o lner sequence for $\Gamma$.

Then, $\mu_{F_n,\gcdinfty}$ converges to $\mu_{\infty,\gcdinfty}^{\Gamma}$ as $n$ goes to infinity.
\end{coro}

\begin{proofthmaff}
Exactly as in Section~\ref{subsec:gcd}, one defines $\mu^\Gamma_{\infty,\supercol}$ and proves that $\mu^\Gamma_{F_n,\supercol}$ converges to $\mu^\Gamma_{\infty,\supercol}$. Let $d(\cdot,\Gamma)$ be as in the conclusion of Lemma~\ref{lem:gcd}. Let $X:=\prod_{x\in \Z^d\backslash \Gamma} \psi(\llbracket 1,d(x,\Gamma)\rrbracket)$. As the closed subset $X \subset Y:= \SuperN^{\Z^d\backslash \Gamma}$ has probability 1 for every measure $\mu^\Gamma_{F_n,\supercol}$, it is the case that $\mu^\Gamma_{\infty,\supercol}(X)=1$. Therefore, one can also consider these measures as probability measures on $X$: with this point of view, denote them by $\mu_n$, with $n\leq \infty$. One has that $\mu_n$ converges to $\mu_\infty$ --- that is for the weak topology on probability measures \emph{on $X$}. One way to see this is to notice that as $X$ is compact, every subsequence of $(\mu_n)$ admits a converging subsequence and then to apply Lemma~\ref{lem:topolike} to the restriction of the identity map from $X$ to $Y$. As the map $f:X\to \llbracket 0,\infty\rrbracket^{\Z^d\backslash \Gamma}$ defined by $f(\sigma)_x:=\prod_{p\in\Prime} p^{\sigma(x)_p}$ is continuous, one has that $\mu_{F_n,\gcdinfty}$ converges to $\mu_{\infty,\gcdinfty}^{\Gamma}$ on $\Z^d\backslash \Gamma$.

Let us now assume that $k=1$. As $\sum_{p\in\Prime}\frac{1}p=\infty$, the second Borel--Cantelli Lemma yields that $\mu^\Gamma_{\infty,\supercol}$-almost every configuration gives the label $\infty$ to every element of $\Gamma$. If $Y_n$ denotes a uniformly chosen element of $F_n$ and $y$ some element of $\Gamma$, as $k$ is equal to $1$, one has that $\gcd(Y_n+y)$ converges in probability to $\infty$. Together with convergence on $\Z^d\backslash \Gamma$, this implies that $\mu_{F_n,\gcdinfty}$ converges to $\mu_{\infty,\gcdinfty}^{\Gamma}$.

Instead of $k=1$, now assume that $\gcd(Y_n)$ is tight. As $(F_n)$ is a F\o lner sequence for $\Gamma$, one has that for every $y\in \Gamma$, the random variable $\gcd(Y_n+y)$ is tight. Together with convergence on $\Z^d\backslash \Gamma$, this implies that the sequence $\left(\mu_{F_n,\gcdinfty}\right)$ is tight. One then concludes as in the proof of Proposition~\ref{prop:gcd}.

\end{proofthmaff}

\begin{lem}
\label{lem:gcd}
Let $(d,k)$ satisfy $1\leq k \leq d$. Let $\Gamma$ be a subgroup of $\Z^d$ of rank $k$ and maximal with this property.
Then, there is a norm $\|~\|$ on $\R^d$ such that for every $N\geq 1$ and every $x\in N\Z^d$, one has
$$
d(x,\Gamma)>0\implies d(x,\Gamma)\geq N,
$$
where $d(x,\Gamma):=\min\{\|x-y\|:y\in \Gamma\}$.
\end{lem}

\begin{proof}
By a change of coordinates and by maximality of $\Gamma$, one may assume that $\Gamma=\Z^k\times\{0_{d-k}\}$, in which case the lemma is clear. \begin{small}Notice that any element of $\mathsf{GL}_d(\Z)$ maps $N\Z^d$ precisely to itself.\end{small}
\end{proof}


\vspace{0.2cm}

\begin{proofcoraff}
Let $Y_n$ denote a uniform element of $F_n$.
By Proposition~\ref{thm:gcdaffine}, it suffices to assume that $k\geq 2$ and to show that $\gcd(Y_n)$ is tight. Since \textsc{gcd}'s are unchanged by $\mathsf{GL}_d(\Z)$, we may assume that $\Gamma$ is equal to $\Z^k\times\{0_{d-k}\}$, and tightness results from (\ref{eq:gcdfolklore}).
\end{proofcoraff}

\subsection{Graphon and local-graphon limits}

This section is devoted to \begin{small}(variations of)\end{small} the following question: if $X_1,\dots, X_N$ are sampled ``uniformly'' and independently in $\Z^d$, can we describe the set of $(i,j)$'s such that $X_i$ is visible from $X_j$, i.e.~such that $\gcd(X_i-X_j)=1$?

\vspace{0.2cm}

A \defini{``graphon''} is the data of a standard probability space $(\mathfrak{X},\mathbb{P})$ together with a measurable function $f:\mathfrak{X}^2\to [0,1]$ that is symmetric, i.e.~satisfies $\forall (x_1,x_2) \in \mathfrak{X}^2,~f(x_1,x_2)=f(x_2,x_1)$. Two ``graphons'' $(\mathfrak{X},\mathbb{P}_\mathfrak{X},f)$ and $(\mathfrak{Y},\mathbb{P}_\mathfrak{Y}, g)$ are said to \defini{induce the same graphon} if, up to throwing away sets of measure zero, there is a measure-preserving isomorphism $\pi:(\mathfrak{X},\mathbb{P}_\mathfrak{X})\to (\mathfrak{Y},\mathbb{P}_\mathfrak{Y})$  such that $\forall (x_1,x_2)\in \mathfrak{X},~f(x_1,x_2)=g(\pi(x_1),\pi(x_2))$. A \defini{graphon} is an equivalence class of ``graphons'' for the relation ``inducing the same graphon''. We say that a graphon is \defini{represented} by any ``graphon'' that induces it. See \cite{graphons}.

Let $\mathcal{G}_n=(V_n,E_n)$ denote a sequence of random\footnote{One does not need $\mathcal{G}_n$ and $\mathcal{G}_m$ to be defined on the same probability space.} finite graphs such that $|V_n|$ converges in probability to infinity. It is said to \defini{converge} to the \begin{small}(deterministic)\end{small} graphon represented by $(\mathfrak{X},\mathbb{P},f)$ if the following holds: for every $k$, if $(X_1^n,\dots,X_k^n)$ denotes a uniform element of $V_n^k$, then the random variable $(\mathds{1}_{\{X_i^n,X_j^n\}\in E_n})_{1\leq i < j \leq k}$ converges in distribution to \begin{small}(the distribution of)\end{small} the random variable $(f(X_i^\infty,X_j^\infty))_{1\leq i < j \leq k}$, where the $X_i^\infty$'s are independent random variables of distribution $\mathbb{P}$. See \cite{diaconisjanson}.

Consider the following standard probability space $\stand:=\prod_{p\in \Prime} (\Z/p\Z)^d$, endowed with the product of uniform measures. Consider the measurable function $\delta: \stand^2\to [0,1]$ defined by $\delta(x_1,x_2):=\mathds{1}_{\forall p,~x_1(p)\not=x_2(p)}$. See Figure~\ref{fig:graphon}.

\begin{figure}[!ht]
\centering
\includegraphics[width=6.5cm]{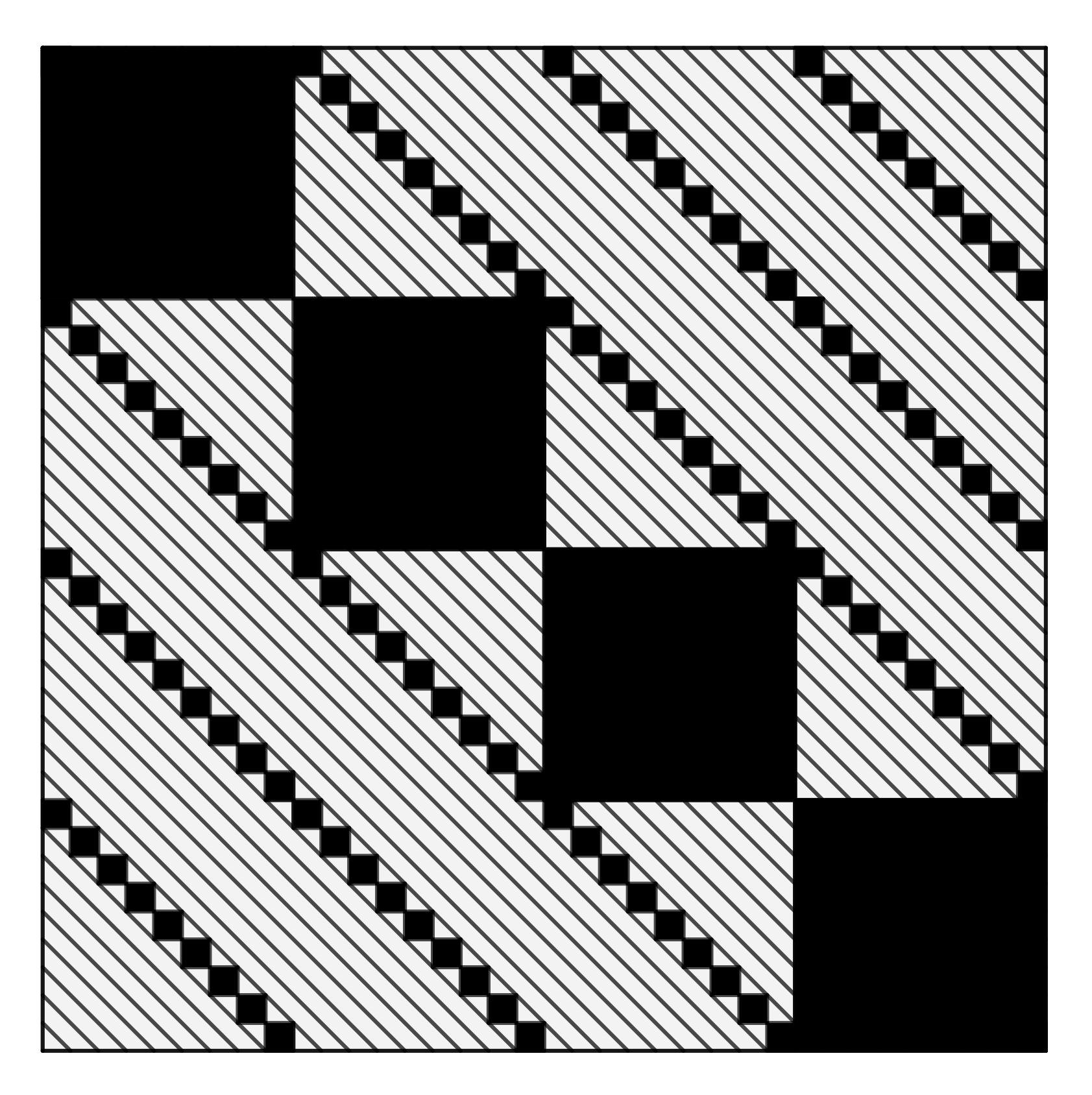}
\footnotesize
\
\caption{Visualisation of the graphon $(\stand,\delta)$ for $d=2$. After suitable identification of $\stand$ with $[0,1]$, the set $\delta^{-1}(\{0\})$ is represented in black. Each square at scale $n$ is divided into $p_n^d\times p_n^d$ smaller squares, where $(p_n)$ denotes the sequence of prime numbers.}
\label{fig:graphon}
\end{figure}

By the same arguments as in Section~\ref{subsec:coprime}, one can prove the following result.

\begin{prop}
\label{prop:graphon}
Let $d\geq 1$ and let $(F_n)$ be a F\o lner sequence of $\Z^d$. Assume that the probability that a uniform element of $F_n$ is coprime converges to $1/\zeta(d)$. Let $\mathcal{G}_n$ denote the random graph with vertex-set $F_n$ and an edge between two distinct vertices if and only if the one is visible from the other.

Then, $\mathcal{G}_n$ converges to the graphon represented by $(\stand,\delta)$.
\end{prop}

Actually, we can even get a result combining both Proposition~\ref{prop:coprime} and Proposition~\ref{prop:graphon}.

\begin{prop}
\label{prop:generalcoprime}
Let $d\geq 1$ and let $(F_n)$ be a F\o lner sequence of $\Z^d$. Assume that the probability that a uniform element of $F_n$ is coprime converges to $1/\zeta(d)$.
Let $(X_n)$ denote a sequence of independent uniform elements of $\stand$.

Let $M\geq 1$ and $R\geq 1$. For every $n$, let $(Y_m^n)_{1\leq m\leq M}$ denote $M$ independent uniform elements in $F_n$. Consider the following random maps:
$$
\begin{array}{lccl}
\psi_n : & (\llbracket 1,M,\rrbracket\times \llbracket -R,R\rrbracket^d)^2&\longrightarrow& \{0,1\} \\
    & ((m_0,{y_0}),(m_1,{y_1})) & \longmapsto & \mathds{1}_{Y_{m_0}^n+{y_0}\text{ is visible from }Y_{m_1}^n+{y_1}},\end{array}
$$
$$
\begin{array}{lccl}
\psi_\infty : & (\llbracket 1,M,\rrbracket\times \llbracket -R,R\rrbracket^d)^2&\longrightarrow& \{0,1\} \\
    & ((m_0,{y_0}),(m_1,{y_1})) & \longmapsto & \mathds{1}_{\forall p,~X_{m_0}(p)+\overline{y_0}\not= X_{m_1}(p)+\overline{y_1}}.\end{array}
$$

Then, the distribution of $\psi_n$ converges to that of $\psi_\infty$.
\end{prop}

This theorem can be readily adapted to the whole \textsc{gcd} profile \begin{small}(one assumes tightness, considers maps to $\llbracket 0,\infty\rrbracket$ rather than to $\{0,1\}$, and predicts the \textsc{gcd} of $(Y_{m_0}^n+{y_0})-(Y_{m_1}^n+{y_1})$)\end{small} and the case of affine subspaces.

Let us conclude with a last generalisation. One may be interested in other arithmetic conditions than coprimality: for example, saying that a number is \defini{$k$-free} if no $p^k$ divides it, one may colour in white the $k$-free points of $\Z$, or the points in $\Z^d$ with $k$-free \textsc{gcd}. See \cite{baakemoodypleasants, cesaro85, mirsky47, mirsky48, pleasantshuck}. One can generalise Proposition~\ref{prop:generalcoprime} to such contexts as follows.
Recall that $\Z$ and $\Zhat$ can be endowed with the profinite topology. Likewise, $\Z^d$ and $\Zhat^d$ can be endowed with the product of profinite topologies, which we also refer to as the profinite topology.

\begin{example}For every $k\geq 0$, the set $\{x\in\Z^d:\gcd(x)\text{ is }k\text{-free}\}$ is profinitely closed in $\Z^d$. For $k=1$, this coincides with the set $\{x\in\Z^d:\gcd(x)=1\}$.
\end{example}

\newcommand{\colv}{\mathsf{col}_{V}}

Let $V$ be a subset of $\Z^d$, and let $\colv:=\mathds{1}_V$. Let $\mu_{\infty,\colv}$ denote the distribution of $x\mapsto \mathds{1}_{Y+x\in \overline{V}}$, where $Y$ is Haar distributed in $\hat{\Z}^d$ and $\overline{V}$ denotes the closure of the set $V$ \emph{in $\hat{\Z}^d$}.

\begin{prop}
Let $d\geq 1$ and let $(F_n)$ be a F\o lner sequence of $\Z^d$. Let $V$ denote a subset of $\Z^d$. Assume that $V$ is profinitely closed \emph{in $\Z^d$}. Assume that $\frac{|V\cap F_n|}{|F_n|}$ converges to $\mu_{\infty,\colv}(\{\omega:\vec{0}\in\omega\})$.
Let $(X_m)$ denote a sequence of independent Haar distributed elements of $\Zhat^d$.

Let $M\geq 1$ and $R\geq 1$. For every $n$, let $(Y_m^n)_{1\leq m\leq M}$ denote $M$ independent uniform elements in $F_n$. Consider the following random maps:
$$
\begin{array}{lccl}
\psi_n : & (\llbracket 1,M,\rrbracket\times \llbracket -R,R\rrbracket^d)^2&\longrightarrow& \{0,1\} \\
    & ((m_0,{y_0}),(m_1,{y_1})) & \longmapsto & \mathds{1}_{(Y_{m_0}^n+{y_0})-(Y_{m_1}^n+{y_1})\in V},\end{array}
$$
$$
\begin{array}{lccl}
\psi_\infty : & (\llbracket 1,M,\rrbracket\times \llbracket -R,R\rrbracket^d)^2&\longrightarrow& \{0,1\} \\
    & ((m_0,{y_0}),(m_1,{y_1})) & \longmapsto & \mathds{1}_{X_{m_0}-X_{m_1}+\profcol(y_0-{y_1})\in \overline{V}}.\end{array}
$$

Then, the distribution of $\psi_n$ converges to that of $\psi_\infty$.
\end{prop}

\paragraph{Acknowledgements} At the end of a talk given by Nathanaël Enriquez, B\'alint Vir\'ag asked him about the local limit of visible points: I would like to thank B\' alint Vir\'ag for asking this question and Nathanaël Enriquez for letting me know of it. I am also grateful to Nathanaël Enriquez for many enthusiastic discussions about this project. I am thankful to Subhajit Goswami and Aran Raoufi for comments on an earlier version of this paper. Finally, I am grateful to my postdoctoral advisors Nicolas Curien and Jean-François Le Gall, as well as to the ERC grant GeoBrown and the Université Paris-Sud, for providing me with an excellent working environment.

\small
\newcommand{\etalchar}[1]{$^{#1}$}

\end{document}